\documentclass[a4paper,12pt]{amsart}
\usepackage{amsfonts}
\usepackage{mathrsfs}
\usepackage{amsmath,amssymb,latexsym,amsfonts,amscd}

\title{Anti-pluricanonical systems on $\bQ$-Fano 3-folds}
\author{Meng Chen}
\date{}

\address{\rm Institute of Mathematics, School of Mathematical Sciences, Fudan University,
Shanghai, 200433, China}
\email{mchen@fudan.edu.cn}

\thanks{The author was
supported by National Outstanding Young Scientist Foundation
(\#10625103) and National Natural Science Foundation of China (Key
Project: \#10731030)}

\newcommand{\bQ}{{\mathbb Q}}
\newcommand{\bP}{{\mathbb P}}
\newcommand{\roundup}[1]{\lceil{#1}\rceil}
\newcommand{\rounddown}[1]{\lfloor{#1}\rfloor}

\newcommand\OO{{\mathcal{O}}}

\newcommand\Lm{{\mathcal{L}_m}}

\newtheorem{thm}{Theorem}[section]
\newtheorem{lem}[thm]{Lemma}
\newtheorem{cor}[thm]{Corollary}
\newtheorem{prop}[thm]{Proposition}

\theoremstyle{definition}
\newtheorem{defn}[thm]{Definition}
\newtheorem{setup}[thm]{}

\newtheorem{rem}[thm]{Remark}
\theoremstyle{remark}

\newtheorem{Prbm}{\bf Problem}
\begin{document}
\begin{abstract} We investigate birationality of the anti-pluricanonical
map $\varphi_{-m}$, the rational map defined by the
anti-pluricanonical system $|-mK_X|$, on weak $\bQ$-Fano 3-folds.
\end{abstract}
\maketitle

\section{\bf Introduction}

Throughout a normal projective variety $X$ is called a {\it terminal
weak $\bQ$-Fano variety} if $X$ has at worst terminal singularities,
$X$ is $\bQ$-factorial and $-K_X$ is a nef and big $\bQ$-Cartier
divisor where $K_X$ is the canonical Weil divisor of $X$. A weak
$\bQ$-Fano variety is said to {\it be $\bQ$-Fano} if $-K_X$ is
$\bQ$-ample.

According to the Minimal Model Program (MMP), $\bQ$-Fano varieties
form a fundamental aspect in birational geometry.

Let $X$ be a terminal weak $\bQ$-Fano 3-fold. The number
$g(X):=h^0(X, \omega_X^{[-1]})-2$ is usually called the genus of
$X$. Denote by $r_X$ the Cartier index of $X$. For a positive
integer $m$, we define $|-mK_X|$ to be the natural closure of the
complete linear system $|-mK_{X^0}|$ where $X^0:=X\backslash
X_{\text{Sing}}$. Whenever $P_{-m}:=h^0(X, \OO_X^{[-m]})\geq 2$, the
rational map $\varphi_{-m}$ corresponding to $|-mK_X|$ gives rise to
the so-called ``anti-canonical geometry'' of $X$. Therefore a
natural interesting question is to find a practical number $m$ such
that $\varphi_{-m}$ is birational onto its image. Such a number $m$
(independent of $X$) exists due to the boundedness theorem which was
proved by Kawamata \cite{KA} (for Picard number one case) and by
Koll\'ar-Miyaoka-Mori-Takagi \cite{KMMT} (for general case). Even
though, it is very interesting to ask the following:

\begin{Prbm}\label{P} Can one find an optimal constant $c$ such that
$\varphi_{-c}$ is birational onto the image for all weak $\bQ$-Fano
3-folds?
\end{Prbm}

When $X$ is Gorenstein (i.e. $r_X=1$), one may take $c=5$ according
to Ando \cite{Ando}. In fact, considerable classification has been
done for Fano 3-folds (see, for example, \cite{I1}, \cite{I2},
\cite{I-P}, \cite{MM}, \cite{Mu1}, \cite{Mu2}, \cite{Nami},
\cite{Shok} and so on). When $X$ is non-Gorenstein, Problem \ref{P}
is open in general.

In contrast to $\varphi_m$ or $|mK_X|$, $\varphi_{-m}$ shows some
pathological nature. For instance, the birationality behavior of
$\varphi_{-m}$ is not birationally invariant! This often makes the
situation be more complicated.

The aim of this note is to build an effective mechanic to study the
birationality of $\varphi_{-m}$ on weak $\bQ$-Fano 3-folds. Our main
technical theorem is {\bf Theorem} \ref{mb}, which has direct
applications to various situations in Section 4 (see Theorems
\ref{3}, \ref{2}, \ref{1}). In fact, since $P_{-8}>1$ by
\cite[Theorem 1.1 (ii)]{C-C}, the rational map
$\varphi_{-8}:X\dashrightarrow \bP^N $ can be properly defined. Thus
the geometry of $X$ can be detected by studying $\varphi_{-8}$. Here
is a direct consequence of our general arguments in Section 4:

\begin{thm}\label{8} Let $X$ be a terminal weak $\bQ$-Fano 3-fold
with $r_X>1$. Then $\varphi_{-m}$ is birational onto its image under
one of the following conditions:
\begin{itemize}
\item[(i)] $\dim\overline{\varphi_{-8}(X)}=3$ and $m\geq 32$;

\item[(ii)] $\dim\overline{\varphi_{-8}(X)}=2$ and $m\geq
\text{max}\{2r_X+16, 48\}$;

\item[(iii)] $\dim\overline{\varphi_{-8}(X)}=1$ and
$$m\geq \begin{cases} 16, & r_X=2;\\
20,& r_X=3;\\
3r_X+10,& r_X\geq 4.
\end{cases}$$
\end{itemize}
\end{thm}

\begin{rem} By virtue of the boundedness theorem for $\bQ$-Fano 3-folds, $r_X$ is upper bounded.
Therefore Theorem \ref{8} has actually obtained a universal constant
$c$ with $\varphi_{-c}$ birational. However such a constant $c$
might still be a little bit far from optimal.
\end{rem}

Restricting our interest to $\bQ$-Fano 3-folds with Picard number
one, what we can prove is slightly favorable. Note that $\bQ$-Fano
3-folds with Picard number one form an important class since the
3-dimensional MMP (see, for instance, \cite{KMM}, \cite{K-M},
\cite{Reid83} and \cite{Sho}) says that any terminal object (after
running MMP) of a given smooth projective variety is either a Mori
fiber space (whose fibers are $\bQ$-Fano with Picard number one) or
a minimal 3-fold. Furthermore, many important works relating to the
classification of non-Gorenstein $\bQ$-Fano 3-folds with Picard
number one have been achieved so far (see \cite{Alex94},
\cite{Prok}, \cite{Suzuki}, \cite{TKG}, \cite{ZQ} and others for a
sample of references). Here we would especially like to mention a
conjecture of Reid: a general member $D\in |-K_X|$ is a K3 surface
for a ``general'' $\bQ$-Fano 3-fold of Picard number one. The
conjecture in Fano case was proved by Reid and Shokurov \cite{Sho2},
although it is still open in general.

In order to make a concise statement of our result, we say that a
$\bQ$-Fano 3-fold $X$ is {\it standard} if
\begin{itemize}
\item[$\cdot$]
the Picard number $\rho(X)=1$;

\item[$\cdot$] $g(X)\geq 0$, a general member $D\in |-K_X|$ is irreducible and
reduced, and a smooth birational model $S$ of $D$ has non-negative
Kodaira dimension, i.e. $\kappa(S)\geq 0$.
\end{itemize}

Roughly speaking, if Reid's conjecture holds true, then a
``general'' terminal $\bQ$-Fano 3-fold with Picard number one is
standard.

Here is our another result:

\begin{thm}\label{standard} Let $X$ be a standard $\bQ$-Fano 3-fold. Then $\varphi_{-m}$ is birational for all $m\geq 6$.
\end{thm}

As mentioned earlier, the linear system $|-mK_X|$ shows a lot of
pathological properties for $\bQ$-Fano 3-folds. To overcome the
obvious difficulties, we will extend the $\bQ$-divisor method that
was used to treat 3-folds of general type. Thanks to the
self-optimization function of our mechanic and to the boundedness
theorem for $\bQ$-Fano 3-folds, we have managed to prove the
effective birationality of $\varphi_{-m}$. Unfortunately, this
method is less effective for those $\bQ$-Fano 3-folds on which, for
some integer $l>0$, $|-lK_X|$ is composed of birationally ruled
surfaces, though that kind of ``bad'' $\bQ$-Fano 3-folds can be
regarded as ``minority''.

This paper is organized as follows. Section 2 mainly sets up the
rational map $\varphi_{-m_0}$ and then reduces the problem to that
on a nonsingular model. In Section 3, we prove our key theorem on
the birationality of $\varphi_{-m}$. Various concrete applications
will be presented in Section 4. In Appendix, we prove $P_{-m}>0$ for
all $m\geq 6$ on any $\bQ$-Fano 3-folds, which is applied to prove
our main theorems.

Sometimes, we use the notation $\OO_X(mK_X)$ to denote the reflexive
sheaf $\omega_X^{[m]}$. The symbol $\sim$ always means linear
equivalence while $\equiv$ represents numerical equivalence.
\bigskip

I would like to thank Jungkai A. Chen, H\'el\`ene Esnault,
Christopher D. Hacon, Yujiro Kawamata, Janos Koll\'ar and De-Qi
Zhang either for stimulating discussions or for their generous
helps.

\section{\bf Preliminaries and the main reduction}

\begin{setup}{\bf Convention.}
Let $X$ be a terminal weak $\bQ$-Fano 3-fold. Pick a canonical
divisor $K_X$ on $X$. Denote by $r_X$ the Cartier index of $X$,
which is nothing but the minimal positive integer with $r_XK_X$
being Cartier. For any positive integer $m$, the number
$P_{-m}:=h^0(X,\OO_X(-mK_X))$ is called the {\it anti-plurigenus} of
$X$. By convention, $g:=h^0(X,\OO_X(-K_X))-2$ is referred to as the
{\it genus} of $X$. Clearly, since $-K_X$ is nef and big, Kawamata's
vanishing theorem \cite[Theorem 1-2-5]{KMM} implies
$$h^i(\OO_X)=h^i(X, K_X-K_X)=0$$
for all $i>0$. Thus $\chi(\OO_X)=1$. Furthermore, the positive
rational number $-K_X^3$ is called the {\it anti-canonical volume}.
\end{setup}

By \cite[Theorem 1.1 (ii)]{C-C}, one has $P_{-8}\geq 2$. Thus, for a
given weak $\bQ$-Fano 3-fold $X$, there exists a positive integer
$m_0\leq 8$ such that $P_{-m_0}\geq 2$. We shall begin from studying
the geometry induced by the rational map $\varphi_{-m_0}$.

\begin{setup}\label{set}{\bf Set up for $\varphi_{-m_0}$.}
We study the $m_0$-th anti-canonical map of $X$:
$$X\overset{\varphi_{-m_0}}{\dashrightarrow} \bP^{P_{-m_0}-1}$$ which is
not necessarily well-defined everywhere. First of all we fix an
effective Weil divisor $\hat{K}_0\sim -m_0K_X$. By Hironaka's big
theorem, we can take successive blow-ups $\pi: Y\rightarrow X$ such
that:
\begin{itemize}
\item [(i)] $Y$ is nonsingular projective;
\item [(ii)] the movable part $|M_{-m_0}|$ of the linear system
$|\rounddown{\pi^*(-m_0K_X)}|$ is base point free and, consequently,
the rational map $\gamma:=\varphi_{-m_0}\circ \pi$ is a morphism;
\item [(iii)] the support of the
union of $\hat{K}_0$ and the exceptional divisors of $\pi$ is of
simple normal crossings (to secure the usage of vanishing theorems).
\end{itemize}
Let $Y\overset{f}\longrightarrow \Gamma\overset{s}\longrightarrow Z$
be the Stein factorization of $\gamma$ with $Z:=\gamma(Y)\subset
\bP^{P_{-m_0}-1}$. In summary, we have the following commutative
diagram:\medskip

\begin{picture}(50,80) \put(100,0){$X$} \put(100,60){$Y$}
\put(170,0){$Z$} \put(170,60){$\Gamma$}
\put(112,65){\vector(1,0){53}} \put(106,55){\vector(0,-1){41}}
\put(175,55){\vector(0,-1){43}} \put(114,58){\vector(1,-1){49}}
\multiput(112,2.6)(5,0){11}{-} \put(162,5){\vector(1,0){4}}
\put(133,70){$f$} \put(180,30){$s$} \put(95,30){$\pi$}
\put(130,-5){$\varphi_{-m_0}$}\put(136,40){$\gamma$}
\end{picture}
\bigskip

Recall that
$$\pi^*(K_X):=K_{Y}-\frac{1}{r_X}E_{\pi}$$ with $E_{\pi}$ effective and
$\bQ$-Cartier since $X$ has at worst terminal singularities.
Equivalently, $-K_Y=\pi^*(-K_X)-\frac{1}{r_X}E_{\pi}$.

{\bf Case $(f_{\text{np}})$.} If $\dim(\Gamma)\geq 2$, a general
member $S$ of $|M_{-m_0}|$ is a nonsingular projective surface by
Bertini's theorem. We say that $|-m_0K_X|$ {\it is not composed with
a pencil of surfaces}.

{\bf Case $(f_\text{p})$.} If $\dim(\Gamma)=1$, then $\Gamma\cong
\bP^1$ since $g(\Gamma)\leq q(Y)=q(X):=h^1(\OO_X)=0$. Furthermore, a
general fiber $S$ of $f$ is an irreducible smooth projective surface
by Bertini's theorem again. We may write
$$M_{-m_0}=\underset{i=1}{\overset{a_{m_0}}\sum}S_i\equiv
a_{m_0}S$$ where $S_i$ is a smooth fiber of $f$ for all $i$ and
$a_{m_0}=P_{-m_0}-1$. In this case, we say that $|-m_0K_X|$ {\it is
composed with a (rational) pencil of surfaces}.

Define
$$\iota:=\begin{cases} 1, & \text{Case\ } (f_{\text{np}})\\
a_{m_0}, & \text{Case\ } (f_{\text{p}}).
\end{cases}$$
\end{setup}
Clearly, at any case, $M_{-m_0}\equiv \iota S$ with $\iota\geq 1$.

\begin{defn} For both Case $(f_{\text{np}})$ and Case $(f_{\text{p}})$, we call $S$
{\it a generic irreducible element of $|M_{-m_0}|$}.
\end{defn}

\begin{lem}\label{L^2} Keep the same notation as above.
The number $$r_X(\pi^*(-K_X)^2\cdot S)_Y$$ is a positive integer.
\end{lem}
\begin{proof} Note that the number $(\pi^*(-K_X)^2\cdot S)$ is
positive since $(\pi^*(-K_X)^2\cdot S)_Y=(\pi^*(-K_X)|_S)_S^2$ and
$\pi^*(-K_X)|_S$ is nef and big on $S$.  It is also independent of
the choice of $\pi$ according to the projection formula of the
intersection theory. So we may choose such a modification $\pi$ that
dominates a resolution of singularities $\tau: \hat{W}\rightarrow
X$. Then we see $(\pi^*(-K_X)^2\cdot S)_Y=(\tau^*(-K_X)^2\cdot
{S_1})_{\hat{W}}$ where $S_1=\delta_*(S)$ is a divisor on $\hat{W}$
and $\delta:Y\rightarrow \hat{W}$ is a birational morphism. Note,
however, $S_1$ is a generic element in an algebraic family though it
is not necessarily nonsingular.

We may write $K_{\hat{W}}=\tau^*(K_X)+\Delta_{\tau}$ where
$\Delta_{\tau}$ is an exceptional effective $\bQ$-divisor over those
isolated terminal singularities on $X$. Now, by intersection theory,
we have $(r\tau^*(-K_X)\cdot \tau^*(-K_X)\cdot
S_1)_{\hat{W}}=(r\tau^*(-K_X)\cdot (-K_Z)\cdot S_1)_{\hat{W}}$ is an
integer.
\end{proof}

The proof for the next lemma was suggested by C. D. Hacon.

\begin{lem}\label{Hn} Let $W$ be a normal projective variety on which there is
an integral Weil $\bQ$-Cartier divisor $D$. Let $h: V\longrightarrow
W$ be any resolution of singularities. Assume that $E$ is an
effective exceptional $\bQ$-divisor on $V$ with $h^*(D)+E$ a Cartier
divisor on $V$. Then
$$h_*\OO_V(h^*(D)+E)=\OO_W(D)$$
where $\OO_W(D)$ is the reflexive sheaf corresponding to the Weil
divisor $D$.
\end{lem}
\begin{proof} By \cite[Lemma 2.11]{Naka},
$h_*\OO_V(\rounddown{h^*(D)})=\OO_W(D).$ Consequently, one has:
$$\OO_W(D)=h_*\OO_V(\rounddown{h^*(D)})\hookrightarrow
h_*\OO_V(h^*(D)+E) \overset{j}\longrightarrow \OO_W(D)$$ where, for
any section $g\in h_*\OO_V(h^*(D)+E)(U)$ over a Zariski open subset
$U\subset W$, $g$ comes from a rational function $\tilde{g}$ on
$h^{-1}(U)$ with
$$((\tilde{g})+(h^*(D)+E))|_{h^{-1}(U)}\geq 0.$$
However, since $\text{Codim} (W_{\text{Sing}})\geq 2$ and by the
valuation theory, one has
$$((g)+D)|_{U}=((g)+D)|_{U\cap W_{\text{Sing}}}=
((\tilde{g})+h^*(D)+E)|_{h^{-1}(U\cap W_{\text{Sing}})}\geq 0.$$
Thus $g$ can be viewed as a section $j(g)$ in $\OO_W(D)(U)$ and $j$
is defined in this way.

Since $j$ is identical over $W^0:=W\backslash W_{\text{Sing}}$ and
$h_*\OO_V(h^*(D)+E)$ is torsion free, $j$ is an inclusion. On the
other hand, pick a rational function $\xi$ on an open set $U$ of $X$
such that $(\xi)+D\geq 0$. Then $(\xi\circ h)+h^*(D)\geq 0$ and so
$(\xi\circ h)+h^*(D)+E\geq 0$. So $\xi$ belongs to
$h_*\OO_V(h^*D+E)(U)$ and $j$ is surjective. Hence $j$ is the
identity.
\end{proof}

\begin{setup}\label{red}{\bf The main reduction.} Let
$\pi:Y\longrightarrow X$ be as in \ref{set}. Let $m>0$ be an
integer. One has:
$$\begin{array}{lll}
&& K_Y+\roundup{(m+1)\pi^*(-K_X)}\\
&=& \pi^*(K_X)+\frac{1}{r}E_{\pi}+\pi^*(-(m+1)K_X)+E_{m+1}\\
&=& \pi^*(-mK_X)+(\frac{1}{r}E_{\pi}+E_{m+1})
\end{array}$$
where $E_{m+1}$ is an effective $\bQ$-divisor on $Y$. Lemma \ref{Hn}
implies $$\pi_*\OO_Y(K_Y+\roundup{(m+1)\pi^*(-K_X)})=\OO_X(-mK_X).$$
Hence $\varphi_{-m}$ is birational if and only if
$\Phi_{|K_Y+\roundup{(m+1)\pi^*(-K_X)}|}$ is birational.

Noting that
$$\begin{array}{lll}
H^0(\OO_X(-mK_X))&\cong& H^0(\OO_Y(\rounddown{-m\pi^*(K_X)})\\
&\cong &H^0(\OO_Y(K_Y+\roundup{(m+1)\pi^*(-K_X)}), \end{array}$$ we
denote by $|M_{-m}|$ the movable part of
$|\rounddown{-m\pi^*(K_X)}|$. Clearly, one has the equality:
$$-m\pi^*(K_X)=M_{-m}+F_{m} \eqno{(2.1)}$$
where $F_m$ is an effective $\bQ$-divisor.
\end{setup}

Another direct consequence is that we may write:
$$K_Y+\roundup{(m+1)\pi^*(-K_X)}\sim M_{-m}+N_{-m} \eqno{(2.2)}$$ where $N_{-m}$
is the fixed part.

\begin{cor}\label{non-pencil} Let $X$ be a terminal weak $\bQ$-Fano 3-fold.
If $$P_{-m_0}>m_0r_X(-K_X)^3+1$$ for some integer $m_0$, then
$|-m_0K_X|$ is not composed with a pencil.
\end{cor}
\begin{proof} Assume $|-m_0K_X|$ is composed with a pencil. Keep the same notation as in \ref{set}.
Then we have $m_0\pi^*(-K_X)\geq M_{-m_0}\equiv a_{m_0}S$ with
$a_{m_0}= P_{-m_0}-1$. Thus one has $m_0(-K_X)^3\geq
(P_{-m_0}-1)(\pi^*(-K_X)^2\cdot S)\geq \frac{1}{r_X}(P_{-m_0}-1)$ by
Lemma \ref{L^2}, a contradiction. We are done.
\end{proof}

\section{\bf The key theorem}

Let $X$ be a terminal weak $\bQ$-Fano 3-fold on which $P_{-m_0}\geq
2$ for some integer $m_0>0$. Keep the same notation as in \ref{set}.
Pick a generic irreducible element $S$ of $|M_{-m_0}|$. Suppose we
have already a base point free linear system $|G|$ on $S$.

\begin{setup}\label{zeta}{\bf Notations.} Denote by
$C$ a generic irreducible element of $|G|$. Since $\pi^*(-K_X)|_S$
is nef and big, there definitely exists a rational number
$\varrho>0$ such that
$$\pi^*(-K_X)|_S-\varrho C$$ is numerically equivalent to en effective
$\bQ$-divisor on $S$. Define
$$\zeta:=(\pi^*(-K_X)\cdot C)_Y=(\pi^*(-K_X)|_S\cdot C)_S$$ which will be the key
quantity accounting for the birationality of $\varphi_{-m}$.
Clearly, since $r_XK_X$ is a Cartier divisor, one has
$$\zeta\geq \frac{1}{r_X}. \eqno{(3.1)}$$
We define two more quantities as follows:
\begin{eqnarray*}
&&\varepsilon:=(m+1-\frac{m_0}{\iota}-\frac{1}{\varrho})\zeta;\\
&&\varepsilon_0:=\roundup{\varepsilon}.
\end{eqnarray*}
\end{setup}

\begin{thm}\label{inequal} Let $X$ be a terminal weak $\bQ$-Fano
3-fold with $P_{-m_0}\geq 2$ for some positive integer $m_0$. Assume
that $|G|$ is a base point free linear system on a generic
irreducible element $S$ of $|M_{-m_0}|$. Let $m>0$ be an integer.
Keep the same notation as above. Then the inequality:
$$\zeta\geq \frac{2g(C)-2+\varepsilon_0}{m}$$
holds under one of the following conditions:
\begin{itemize}
\item[(i)] $g(C)>0$ and $\varepsilon>1$;

\item[(ii)] $g(C)=0$ and $\varepsilon>2$.
\end{itemize}
\end{thm}
\begin{proof} Equality (2.1) gives:
$$m_0\pi^*(-K_X)=M_{-m_0}+F_{m_0}$$ for some effective
$\bQ$-divisor $F_{m_0}$. For the given integer $m>0$, one has:
$$|K_Y+\roundup{(m+1)\pi^*(-K_X)-\frac{1}{\iota}F_{m_0}}|\subset
|K_Y+\roundup{(m+1)\pi^*(-K_X)}|.\eqno{(3.2)}$$ Under the assumption
$\varepsilon>0$, the $\bQ$-divisor
$$(m+1)\pi^*(-K_X)-\frac{1}{\iota}F_{m_0}-S\equiv (m+1-\frac{m_0}{\iota})\pi^*(-K_X)$$
is nef and big and thus $H^1(Y,
K_Y+\roundup{(m+1)\pi^*(-K_X)-\frac{1}{\iota}F_{m_0}}-S)=0$ by
Kawamata-Viehweg vanishing theorem \cite{KV,V}. Hence one has the
surjective map:
$$H^0(Y,K_Y+\roundup{(m+1)\pi^*(-K_X)-\frac{1}{\iota}F_{m_0}})\longrightarrow
H^0(S, K_S+L_m) \eqno{(3.3)}$$ where
$$L_m:=(\roundup{(m+1)\pi^*(-K_X)-\frac{1}{\iota}F_{m_0}}-S)|_S\geq
\roundup{\mathcal{L}_m}$$
and ${\mathcal
L}_m:=((m+1)\pi^*(-K_X)-\frac{1}{\iota}F_{m_0}-S)|_S.$ On the other
hand, we have
$$\pi^*(-K_X)|_S\equiv \varrho C+H$$
for an effective $\bQ$-divisor $H$ on $S$. Thus the $\bQ$-divisor
$$\Lm-\frac{1}{\varrho}H-C\equiv \varepsilon \pi^*(-K_X)|_S$$
is nef and big and, by the vanishing theorem again, $$H^1(S,
K_S+\roundup{\Lm-\frac{1}{\varrho}H}-C)=0.$$ Therefore, one has the
following surjective map:
$$H^0(S, K_S+\roundup{\Lm-\frac{1}{\varrho}H})\longrightarrow
H^0(C, K_C+D_m) \eqno{(3.4)}$$ where
$$D_m:=\roundup{\Lm-\frac{1}{\varrho}H-C}|_C\geq
\roundup{\mathcal{D}_m}$$ and
$\mathcal{D}_m:=(\Lm-\frac{1}{\varrho}H-C)|_C$ with
$\deg\roundup{\mathcal{D}_m}\geq \varepsilon_0$.

Now under the situations (i) and (ii),
$|K_C+\roundup{\mathcal{D}_m}|$ is base point free with
$$\deg(K_C+\roundup{\mathcal{D}_m})\geq 2g(C)-2+\varepsilon_0.$$

Denote by $\mathcal{N}_m$ the movable part of
$|K_S+\roundup{\Lm-\frac{1}{\varrho}H}|$. Noticing the relations
(3.2), (3.3) and $|K_S+\roundup{\Lm-\frac{1}{\varrho}H}|\subset
|K_S+\roundup{\Lm}|$ while applying \cite[Lemma 2.7]{Camb}, one
gets:
$$m\pi^*(-K_X)|_S\geq {M_{-m_0}}|_S\geq \mathcal{N}_m$$
and ${\mathcal{N}_m}|_C=K_C+\roundup{\mathcal{D}_m}$ simply because
the later one is base point free under either of the conditions.

So one has $m\zeta\geq \deg(K_C+\roundup{\mathcal{D}_m})\geq
2g(C)-2+\varepsilon_0$. We are done.
\end{proof}

Riemann-Roch theorem on $C$ directly implies the following:

\begin{cor}\label{nonv} Under the same assumption as that of
Theorem \ref{inequal}, then one has:
\begin{itemize}
\item[(1)] $P_{-m}>0$ whenever $g(C)=0$ and $\varepsilon>1$;

\item[(2)] $P_{-m}>0$ whenever $g(C)>0$ and $\varepsilon>0$.
\end{itemize}
\end{cor}

{}From now on, we shall work on the birationality. We shall always
require that the linear system
$\Lambda_m:=|K_Y+\roundup{(m+1)\pi^*(-K_X)}|$ satisfies the
following assumption for some integer $m>0$.

\begin{setup}\label{asum}{\bf Assumption}: \begin{itemize}
\item [(1)] The linear system $\Lambda_m$
distinguishes different generic irreducible elements of $|M_{-m_0}|$
(namely, $\Phi_{\Lambda_m}(S')\neq \Phi_{\Lambda_m}(S'')$ for two
different generic irreducible elements $S'$, $S''$ of $|M_{-m_0}|$).

\item [(2)] The linear system ${\Lambda_m}_{|S}$ distinguishes different
generic irreducible elements of the given linear system $|G|$ on
$S$.
\end{itemize}
\end{setup}

Here is our key theorem:

\begin{thm}\label{mb} Let $X$ be a terminal weak $\bQ$-Fano
3-fold with $P_{-m_0}\geq 2$ for some positive integer $m_0$. Assume
that $|G|$ is a base point free linear system on a generic
irreducible element $S$ of $|M_{-m_0}|$. Let $m>0$ be an integer. If
Assumption \ref{asum} is satisfied and  $\varepsilon>2$, then
$\varphi_{-m}$ is birational onto its image.
\end{thm}
\begin{proof} Since Assumption \ref{asum} (1) is satisfied, the usual birationality principle
reduces the birationality of $\varphi_{-m}$ to that of
${\Phi_{\Lambda_m}}|_S$ for a generic irreducible element $S$.
Similarly, due to Assumption \ref{asum} (2), one only needs to prove
the birationality of ${\Phi_{\Lambda_m}}|_C$ for a generic
irreducible element $C$ of $|G|$.

Now since one has the surjective maps (3.3) and (3.4), it is
sufficient to prove that $|K_C+\roundup{\mathcal{D}_m}|$ gives a
birational map. Clearly this is the case whenever $\varepsilon>2$,
which in fact implies $\deg(\roundup{\mathcal{D}_m})\geq 3$. We are
done.
\end{proof}

While applying Theorem \ref{mb}, one needs to verify Assumption
\ref{asum} in advance. Hereby we are working on this.

\begin{prop}\label{a1} Let $X$ be a terminal weak $\bQ$-Fano
3-fold with $P_{-m_0}\geq 2$ for some positive integer $m_0$. Then
Assumption \ref{asum} (1) is satisfied for all
$$m\geq \begin{cases}
m_0+6, & m_0\geq 2\\
2, &m_0=1.
\end{cases}$$
\end{prop}
\begin{proof} By Corollary \ref{ap} in Appendix, one has
$$\begin{array}{lll}
&& K_Y+\roundup{(m+1)\pi^*(-K_X)}\\
&\geq & K_Y+\roundup{(m-m_0+1)\pi^*(-K_X)+M_{-m_0}}\\
&=& (K_Y+\roundup{(m-m_0+1)\pi^*(-K_X)})+M_{-m_0}\\
&\geq& M_{-(m-m_0)}+M_{-m_0}\geq M_{-m_0}
\end{array}$$
for all $m-m_0\geq 6$ whenever $m_0\geq 2$ (resp. $\geq 1$ whenever
$m_0=1$). When $f:Y\rightarrow \Gamma$ is of type $(f_{\text{np}})$,
\cite[Lemma 2]{T} implies that $\Lambda_m$ can distinguish different
generic irreducible elements of $|M_{-m_0}|$. When $f$ is of type
$(f_{\text{p}})$, since the rational (i.e. $\Gamma\cong \bP^1$)
pencil $|M_{-m_0}|$  can already separate different fibers of $f$,
$\Lambda_m$ can naturally distinguish different generic irreducible
elements of $|M_{-m_0}|$.
\end{proof}

It is slightly more complicated to verify Assumption \ref{asum} (2).

\begin{lem}\label{S2} Let $T$ be a nonsingular projective surface
on which there is a base point free linear system $|G|$. Let $Q$ be
an arbitrary $\bQ$-divisor on $T$. Then the linear system
$|K_T+\roundup{Q}+G|$ can distinguish different generic irreducible
elements of $|G|$ under one of the following conditions:
\begin{itemize}
\item[(i)] $K_T+\roundup{Q}$ is effective and $|G|$ is not
composed with an irrational pencil of curves;

\item[(ii)] $Q$ is nef and big, $g(C)\geq 0$ and $|G|$ is composed with an
irrational pencil of curves;

\item[(iii)] $Q$ is nef and big, $Q\cdot G>1$, $g(C)=0$ and $|G|$ is composed with an
irrational pencil of curves.
\end{itemize}\end{lem}

\begin{proof} The statement corresponding to (i) follows from \cite[Lemma 2]{T} and
the fact that a rational pencil can automatically separate its
different generic irreducible elements.

For situations (ii) and (iii), we pick a generic irreducible element
$C$ of $|G|$. Then, since $h^0(S, G)\geq 2$, $G\equiv sC$ for some
integer $s\geq 2$ and $C^2=0$. Denote by $C'$ another generic
irreducible element of $|G|$. The Kawamata-Viehweg vanishing theorem
gives the surjective map:
$$H^0(T,K_T+\roundup{Q}+G)\longrightarrow H^0(C,K_C+D)\oplus
H^0(C',K_{C'}+D')$$ where $D:=(\roundup{Q}+G-C)|_C$ and
$D':=(\roundup{Q}+G-C')|_{C'}$ with $\deg(D)\geq Q\cdot C>0$,
$\deg(D')\geq Q\cdot C'>0$.

Whenever $g(C)\geq 0$, Riemann-Roch formula gives
$h^0(C,K_C+D)=h^0(C',K_{C'}+D')>0$. Thus $|K_T+\roundup{Q}+G|$ can
distinguish $C$ and $C'$.

Whenever $g(C)=0$ and $Q\cdot C>1$, $H^0(C, K_C+D)\neq 0$ and
$H^0(C', K_{C'}+D')\neq 0$. So $|K_T+\roundup{Q}+G|$ can also
distinguish $C$ and $C'$.
\end{proof}

\begin{prop}\label{a2} Let $X$ be a terminal weak $\bQ$-Fano
3-fold with $P_{-m_0}\geq 2$ for some positive integer $m_0$. Keep
the same notation as in \ref{set}. Take $|G|=|{M_{-m_0}}_{|S}|$ on a
generic irreducible element $S$ of $|M_{-m_0}|$. If $f:Y\rightarrow
Z$ is of type $(f_{\text{np}})$, then Assumption \ref{asum} (2) is
satisfied under one of the following conditions:
\begin{itemize}
\item[(1)] $g(C)> 0$ and
$m\geq \begin{cases} \max\{m_0+6,2m_0\} & m_0\geq 2\\
2, & m_0=1;
\end{cases}$
\item[(2)] $g(C)=0$, $\varepsilon>1$ and
$m\geq \begin{cases} m_0+6, & m_0\geq 2\\
2, & m_0=1.
\end{cases}$
\end{itemize}
\end{prop}
\begin{proof} By the relation (3.2) and the surjective map (3.3), we
see ${\Lambda_m}|_S\supset |K_S+L_m|$, which will be proved to be
able to distinguish different generic irreducible elements of $|G|$.
By definition, we have $\iota=1$. Since $m_0\pi^*(-K_X)\geq
M_{-m_0}$, we may take $\varrho=\frac{1}{m_0}$.

Now one has the following:
$$\begin{array}{lll}
K_S+L_m
&=& (K_Y)|_S+\roundup{(m+1)\pi^*(-K_X)-F_{m_0}}|_S\\
&\geq & (K_Y+\roundup{(m-m_0+1)\pi^*(-K_X)})|_S+{M_{-m_0}}|_S.
\end{array}$$
Thus, if $|G|$ is not composed with an irrational pencil of curves,
$|K_S+L_m|$ can distinguish different generic irreducible elements
provided that $m-m_0\geq 6$ whenever $m_0\geq 2$ (or that $m\geq 2$
whenever $m_0=1$).

Suppose $|G|$ is composed with an irrational pencil of curves. One
has:
$$\begin{array}{lll}
K_S+L_m&\geq & K_S+\roundup{\Lm}\\
&=& K_S+\roundup{((m-m_0+1)\pi^*(-K_X)-F_{m_0})|_S}\\
&\geq& K_S+\roundup{(m-2m_0+1)\pi^*(-K_X)|_S}+{M_{-m_0}}|_S.
\end{array}$$
If $g(C)>0$, Lemma \ref{S2} (ii) implies that Assumption \ref{asum}
(2) is satisfied for $m\geq 2m_0$. If $g(C)=0$, by Lemma \ref{S2}
(iii), we need the condition $\varepsilon=(m-2m_0+1)\zeta=Q\cdot
C>1$, where we may take $Q=(m-2m_0+1)\pi^*(-K_X)|_S$.
\end{proof}

When $f$ is of type $(f_{\text{p}})$, we have the similar treatment.

\begin{prop}\label{aa2} Let $X$ be a terminal weak $\bQ$-Fano
3-fold. Assume that $P_{-m_0}\geq 2$ for some positive integer $m_0$
and that $|-m_1K_X|$ is not composed with a pencil of surfaces for
another integer $m_1\geq m_0$. Keep the same notation as in
\ref{set}. Take $|G|=|{M_{-m_1}}_{|S}|$ on a generic irreducible
element $S$ of $|M_{-m_0}|$. Let $C$ be a generic irreducible
element of $|G|$. If $f:Y\rightarrow Z$ is of type $(f_{\text{p}})$,
then Assumption \ref{asum} (2) is satisfied under one of the
following conditions:
\begin{itemize}
\item[(1)] $g(C)> 0$, $m\geq \begin{cases} \max\{m_1+6,m_0+m_1\} & m_0\geq 2\\
m_1+1, & m_0=1;
\end{cases}$

\item[(2)] $g(C)=0$, $\varepsilon>1$ and
$m\geq \begin{cases} m_1+6, & m_0\geq 2\\
m_1+1, & m_0=1.
\end{cases}$
\end{itemize}
\end{prop}
\begin{proof} First, we may re-modify our $\pi$ in \ref{set} such
that the movable part $|M_{-m_1}|$ of $|\rounddown{\pi^*(-m_1K_X)}|$
is base point free too. Since $m_1\pi^*(-K_X)\geq M_{-m_1}$, we may
take $\varrho=\frac{1}{m_1}$. Similar to the proof of Proposition
\ref{a2}, it suffices to prove that $|K_S+L_m|$ can distinguish
different generic irreducible elements of $|G|$.

For a suitable integer $m>0$, one has the following:
$$\begin{array}{lll}
K_S+L_m
&=& (K_Y)|_S+\roundup{(m+1)\pi^*(-K_X)-F_{m_0}}|_S\\
&\geq& (K_Y+\roundup{(m-m_1+1)\pi^*(-K_X)})|_S+{M_{-m_1}}|_S.
\end{array}$$
Thus, if $|G|$ is not composed with an irrational pencil of curves,
$|K_S+L_m|$ can distinguish different irreducible elements provided
that $K_Y+\roundup{(m-m_1+1)\pi^*(-K_X)}$ is effective, i.e.
$m-m_1\geq 6$ whenever $m_0\geq 2$ (or $m\geq m_1+1$ whenever
$m_0=1$).

Assume $|G|$ is composed with an irrational pencil of curves. One
has:
$$\begin{array}{lll}
K_S+L_m&\geq & K_S+\roundup{\Lm}\\
&=& K_S+\roundup{((m-m_0+1)\pi^*(-K_X)-F_{m_0})|_S}\\
&\geq& K_S+\roundup{(m-m_0-m_1+1)\pi^*(-K_X)|_S}+{M_{-m_1}}|_S.
\end{array}$$
If $g(C)> 0$, Lemma \ref{S2}(ii) implies that Assumption \ref{asum}
(2) is satisfied for $m\geq m_0+m_1$. If $g(C)=0$, by Lemma
\ref{S2}(iii), we need the condition
$\varepsilon=(m-m_0-m_1+1)\zeta=Q\cdot C>1$, where we may take
$Q=(m-m_0-m_1+1)\pi^*(-K_X)|_S$. We are done.
\end{proof}

\section{\bf Applications}

In this section, we shall present some concrete applications on the
birationality of $\varphi_{-m}$ using Theorem \ref{mb}. In practice,
we take the smallest integer $\hat{m}_0:=\hat{m}_0(X)$ with
$P_{-\hat{m}_0}\geq 2$. Such a number $\hat{m}_0$ is uniquely
determined by $X$ and $\hat{m}_0\leq 8$ by \cite[Theorem 1.1
(ii)]{C-C}.

\begin{thm}\label{3} Let $X$ be a terminal weak $\bQ$-Fano
3-fold. Assume $\dim \overline{\varphi_{-\hat{m}_0}(X)}=3$. Then
$\varphi_{-m}$ is birational onto its image for all $m\geq
4\hat{m}_0$. In particular, $|-mK_X|$ gives a birational map for
$m\geq 32$.
\end{thm}
\begin{proof} Keep the same notation as in \ref{set}. We have
$\iota=1$. Take $|G|=|M_{-\hat{m}_0 |S}|$ on a generic irreducible
element $S$ of $|M_{-\hat{m}_0}|$. Clearly, we may take
$\varrho=\frac{1}{\hat{m}_0}$. Since, by assumption,
$\Phi_{|M_{-\hat{m}_0 |S}|}$ maps a generic irreducible element $C$
of $|G|$ onto a curve, Riemann-Roch formula and Clifford's theorem
on $C$ gives
$$\hat{m}_0\pi^*(-K_X)\cdot C\geq
{M_{-\hat{m}_0}}_{|S}\cdot C\geq \begin{cases} 2,  &g(C)>0\\
1,& g(C)=0.
\end{cases} \eqno{(4.1)}
$$
 We prove the theorem according to the value of
$g(C)$.

Assume $g(C)>0$. We have seen $\zeta\geq \frac{2}{\hat{m}_0}$ by
inequality (4.1). By Proposition \ref{a1} and Proposition \ref{a2},
Assumption
\ref{asum} is satisfied for $$m\geq \begin{cases} \max\{\hat{m}_0+6,2\hat{m}_0\} & \hat{m}_0\geq 2\\
2, & \hat{m}_0=1.
\end{cases} $$
Take any $m\geq 3\hat{m}_0$. Then $\varepsilon\geq
2+\frac{2}{\hat{m}_0}>2$. Thus Theorem \ref{mb} implies that
$\varphi_{-m}$ is birational for all $m\geq 3\hat{m}_0$ and
$\hat{m}_0\neq 2$. For $\hat{m}_0=2$, we need to set $m\geq 8$ to
satisfy $m\geq \hat{m}_0+6$.

Assume $g(C)=0$. We have already seen $\zeta\geq
\frac{1}{\hat{m}_0}$ by inequality (4.1). Take any $m\geq
4\hat{m}_0$. Then $\varepsilon>2$. Again Propositions \ref{a1},
\ref{a2} imply that $\varphi_{-m}$ is birational for all $m\geq
4\hat{m}_0$.
\end{proof}

\begin{thm}\label{2} Let $X$ be a terminal weak $\bQ$-Fano
3-fold with $r_X>1$. Assume $\dim
\overline{\varphi_{-\hat{m}_0}(X)}=2$. Denote by $C$ a generic
irreducible element of $|G|:=|M_{-\hat{m}_0|S}|$. Then
$\varphi_{-m}$ is birational onto its image under one of the
following conditions:
\begin{itemize}
\item[(1)] $g(C)> 0$, $m\geq 6\hat{m}_0$;

\item[(2)] $g(C)=0$, $m\geq 2\hat{m}_0+2r_X$.
\end{itemize}
\end{thm}
\begin{proof} Again we have $\iota=1$ and may take $\varrho=\frac{1}{\hat{m}_0}$.
Clearly $|G|$ is composed with a pencil of curves.

(1) Assume $g(C)>0$. If we take a sufficiently large
$m=m'>4\hat{m}_0$ such that $\varepsilon=(m'+1-2\hat{m}_0)\zeta>1$
(thus $\varepsilon_0\geq 2$), then Theorem \ref{nonv} implies
$\zeta\geq \frac{2}{m'}$. Take $m=m'-1$. Then $\varepsilon\geq
(m'-2\hat{m}_0)\frac{2}{m'}>1$. Theorem \ref{nonv} again gives
$\zeta\geq \frac{2}{m'-1}$. This means, inductively, that $\zeta\geq
\frac{2}{4\hat{m}_0}=\frac{1}{2\hat{m}_0}$.

Whenever $m\geq 6\hat{m}_0$, then Assumption \ref{asum} is naturally
satisfied by Propositions \ref{a1} and \ref{a2}. On the other hand,
$\varepsilon\geq (6\hat{m}_0+1-2\hat{m}_0)\frac{1}{2\hat{m}_0}>2$.
Theorem \ref{mb} implies that $\varphi_{m}$ is birational.

(2) Assume $g(C)=0$. Since we have $\zeta\geq \frac{1}{r_X}$, we may
take $m\geq 2\hat{m}_0+2r_X$ and then $\varepsilon\geq
(m+1-2\hat{m}_0)\zeta>2$. When $r_X>1$, Theorem \ref{mb},
Propositions \ref{a1} and \ref{a2} imply that $\varphi_m$ is
birational for all $m\geq 2\hat{m}_0+2r_X$.
\end{proof}

Especially, when $\hat{m}_0=1$, we have the following:

\begin{cor} Let $X$ be a standard $\bQ$-Fano
3-fold. Then $\varphi_{-m}$ is birational for all $m\geq 6$.
\end{cor}
\begin{proof} We have $\hat{m}_0=1$. According to
\cite[Theorem 2.18]{Alex94}, $|-K_X|$ has no fixed parts and is not
composed with a pencil of surfaces. Since $X$ is standard, $S$ is
not fibred by rational curves. Thus we have $g(C)>0$. Now the result
follows directly from Theorem \ref{3} and Theorem \ref{2} (1).
\end{proof}

{}From now on within this section, we study the most difficult case:
$\dim \overline{\varphi_{-\hat{m}_0}(X)}=1$. {}First, we need to
look for a number $m_1>0$ such that $|-m_1K_X|$ is not composed with
a pencil of surfaces.

\begin{prop}\label{np} Let $X$ be a terminal weak $\bQ$-Fano 3-fold with $r_X>1$.
Then $|-m_1K_X|$ is not composed with a pencil of surfaces under one
of the following situations:
\begin{itemize}
\item[(1)] $r=2$, $m_1=4$;
\item[(2)] $r=3$, $m_1=6$;
\item[(3)] $r\geq 4$, $m_1=r+2$.
\end{itemize}
\end{prop}
\begin{proof} Let $\rho: V\rightarrow X$ be a resolution of
singularities of $X$. According to Lemma \ref{Hn} and
Kawamata-Viehweg Vanishing theorem, for all positive integer $m$,
$$\begin{array}{lll}
P_{-m}:=P_{-m}(X)&=&h^0(V,K_V+\roundup{(m+1)\rho^*(-K_X)})\\
&=&\chi(V, \OO_V(K_V+\roundup{(m+1)\rho^*(-K_X)})). \end{array}$$
Write $K_V=\rho^*(K_X)+\Delta$ where $\Delta$ is an effective
$\bQ$-divisor and is $\rho$-exceptional. Thus $\rho(\Delta)$ is the
set of singularities on $X$. One gets $-K_V=\rho^*(-K_X)-\Delta$.
For any given $m>0$, write $m=nr+t$ with $n\geq 0$, $0\leq t<r=r_X$.
By Riemann-Roch formula on $V$ (\cite[P. 437]{Hartshorne}) and
setting $D:=K_V+\roundup{(m+1)\rho^*(-K_X)}$, one has:
$$\begin{array}{lll}
P_{-m}&=& \frac{1}{12}D\cdot (D-K_V)\cdot
(2D-K_V)+\frac{1}{12}D\cdot c_2(V)+\chi(\OO_V)\\
&=& \frac{1}{12}m(m+1)(2m+1)(-K_X)^3+A\cdot n+C(t)
 \end{array}$$
where $A:=\frac{r}{12}\rho^*(-K_X)\cdot c_2(V)$ and $C(t)$ is a
constant only depending on $t$. Fix a number $t<r$,
$Q_t(n):=P_{-(nr+t)}$ can be viewed as a polynomial in terms of $n$.
Then we have
$$Q_t(n)=\hat{Q}_t(n)+A\cdot n+C(t)$$
where $\hat{Q}_t(n)=\frac{1}{12}(nr+t)(nr+t+1)(2nr+2t+1)(-K_X)^3$.

For $n\geq 1$ or $n=0$ and $t\geq 2$, one has:
$$Q_t(n)\geq 0 \eqno{(4.2)}$$

For $n\leq -1$, by \cite[Corollary 5.25]{K-M},
$\OO_X((nr+t)K_X):=\omega_X^{[nr+t]}$ is Cohen-Macaulay and thus
Serre duality gives:
$$\begin{array}{lll}
\chi(X, \OO_X((nr+t)K_X))&=&-\chi(X, \mathcal{H}{\it
om}(\OO_X((nr+t)K_X),
\omega_X))\\
&=& -\chi(X, \OO_X((1-nr-t)K_X)) \end{array}$$ where the last
equality is due to the fact that $\mathcal{H}{\it
om}(\OO_X((nr+t)K_X), \omega_X)$ is reflexive (cf. \cite[P. 150,
Lemma 1.1.12]{OSS}). Now the vanishing theorem \cite[Theorem
1-2-5]{KMM} gives:
$$\begin{array}{lll}
Q_t(n)&=&\chi(\OO_X((nr+t)K_X))\\
&=&-\chi(\OO_X((1-nr-t)K_X))=-P_{1-nr-t} \leq 0
\end{array}\eqno{(4.3)}
$$

Let us estimate the lower bounds of both $A$ and $C(t)$. For any
$r$, $t$, one has, by inequalities (4.2) and (4.3),
$$A+C(t)\geq -\hat{Q}_t(1)\geq -\frac{1}{12}(r+t)(r+t+1)(2r+2t+1)(-K_X)^3 \eqno{(4.4)}$$
$$A-C(t)\geq \hat{Q}_t(-1)\geq \frac{1}{12}(-r+t)(-r+t+1)(-2r+2t+1)(-K_X)^3 \eqno{(4.5)}$$
which imply:
$$A\geq  -\frac{1}{12}(2r^2+6t^2+6t+1)r(-K_X)^3.  \eqno{(4.6)}$$
\medskip

{\bf Case 1}. Assume $t\geq 2$.

Clearly, one has $r\geq 3$. Then $Q_t(0)=P_{-t}\geq 0$ gives:
$$C(t)\geq -\frac{1}{12}t(t+1)(2t+1)(-K_X)^3.$$
Thus (4.5) implies:
$$\begin{array}{lll}
A&\geq &\frac{1}{12}\{-2r^2+(6t+3)r-(6t^2+6t+1)\}r(-K_X)^3\\
&=& \frac{1}{12}\{-6t^2+(6r-6)t+(-2r^2+3r-1)\}r(-K_X)^3
\end{array}$$ for all $2\leq t<r$. Noting that $A$ is a constant
independent of $t$, one should have:
$$A\geq \begin{cases}
\frac{1}{12}(-\frac{1}{2}r^2+\frac{1}{2})r(-K_X)^3, & \text{taking }  t=\frac{r-1}{2} \text{ when $r\geq 5$ is odd}\\
\frac{1}{12}(-\frac{1}{2}r^2-1)r(-K_X)^3, & \text{taking }
t=\frac{r}{2} \text{ when $r\geq 4$ is even}.
\end{cases}$$
With the above inequalities, we are able to bound $P_{-m}$ from
below. In fact, we have
$$P_{-(r+2)}=Q_2(1)\geq \frac{1}{8}(r^2+10r+24)r(-K_X)^3.$$
Clearly, $P_{-(r+2)}>(r+2)r(-K_X)^3+1$ whenever $r\geq 4$. Thus
Lemma \ref{non-pencil} implies that $|-(r+2)K_X|$ is not composed
with a pencil when $r\geq 4$ and this proves (3).
\medskip

%

{\bf Case 2}. Assume $t=0$.

Inequalities (4.4) and (4.6) give:
$$P_{-nr}\geq \frac{1}{12}(n^2-1)(2nr+3)r^2(-K_X)^3.$$
Clearly $P_{-2r}>2r\cdot r(-K_X)^3+1$ for $r=2,3$. Thus Lemma
\ref{non-pencil} implies that $|-2rK_X|$ is not composed with a
pencil when $r=2$ or $3$. This proves (1) and (2).

Besides, it is easy to get the lower bound of $P_{-(nr+t)}$ for a
given pair $(n,t)$.
\end{proof}

\begin{thm}\label{1} Let $X$ be a terminal weak $\bQ$-Fano
3-fold with $r_X>1$. Assume $\dim
\overline{\varphi_{-\hat{m}_0}(X)}=1$ and $|-m_1K_X|$ is not
composed with a pencil for another integer $m_1>0$. Then
$\varphi_{-m}$ is birational onto its image for all $m\geq
\hat{m}_0+m_1+2r_X$.
\end{thm}
\begin{proof} By definition, we have $\iota\geq 1$ and $m_1\geq
\hat{m}_0$.

Re-modify our original $\pi$ in \ref{set} such that the movable part
$M_{-m_1}$ of the linear system $|\rounddown{\pi^*(-m_1K_X)}|$ is
base point free too. Pick a generic irreducible element $S$ of
$|M_{-\hat{m}_0}|$ and take $|G|:=|M_{-m_1|S}|$. Then we have
$\varrho=\frac{1}{m_1}$. Recall that we have seen $\zeta\geq
\frac{1}{r_X}$.

Assume, for an integer
$$m\geq \begin{cases} \max\{m_1+6,\hat{m}_0+m_1\} & \hat{m}_0\geq 2\\
m_1+1, & \hat{m}_0=1,
\end{cases}$$ one has already $\varepsilon>1$.
Proposition \ref{a1} and Proposition \ref{aa2} imply that Assumption
\ref{asum} is satisfied. In practice, take any $m\geq
\hat{m}_0+m_1+2r$, which naturally satisfies the above requests.
Furthermore $\varepsilon\geq (m+1-\hat{m}_0-m_1)\zeta>2$.
Consequently, Theorem \ref{mb} implies that $\varphi_{-m}$ is
birational for $m\geq \hat{m}_0+m_1+2r_X$.
\end{proof}

\begin{cor}\label{11} Let $X$ be a terminal weak $\bQ$-Fano
3-fold with $r_X>1$. Assume $\dim
\overline{\varphi_{-\hat{m}_0}(X)}=1$. Then $\varphi_{-m}$ is
birational for
$$ m\geq \begin{cases} \hat{m}_0+8, & r=2;\\
\hat{m}_0+12,& r=3;\\
\hat{m}_0+3r+2,& r\geq 4.
\end{cases}$$
\end{cor}
\begin{proof}
Set
$$m_1:=\begin{cases}  2r,& \text{if } r=2,3;\\
r+2,& \text{otherwise}.
\end{cases}$$
The statement follows directly from Proposition \ref{np} and Theorem
\ref{1}.
\end{proof}

If one does not request the ``stable'' birationality, then the
slightly better result can be obtained as follows.

\begin{thm} Let $X$ be a weak $\bQ$-Fano 3-fold with $r_X>1$. Then
$\varphi_{-(\hat{m}_0+3r-1)}$ is birational onto its image.
\end{thm}
\begin{proof} Take $m=\hat{m}_0+3r-1$. When $r_X\geq 3$, we have
$P_{-(m-\hat{m}_0)}>0$ by Corollary \ref{ap}. When $r_X=2$, we have
$P_{-(m-\hat{m}_0)}=P_{-5}>0$ by Proposition \ref{np}. Thus the
proof of Proposition \ref{a1} implies that
$|K_Y+\roundup{(m+1)\pi^*(-K_X)}|$ can distinguish different
irreducible elements of $|M_{-\hat{m}_0}|$. Thus we only need to
prove the birationality of $|K_S+\roundup{\mathcal{L}_m}|$ for a
generic $S$. Now since
$$\mathcal{L}_m= 3r\pi^*(-K_X)|_S=3L$$
and $L:=r\pi^*(-K_X)|_S$ is a nef and big Cartier divisor on $S$
with $L^2\geq r\cdot r(\pi^*(-K_X)^2\cdot S)\geq 2$ (see Lemma
\ref{L^2}), Reider's theorem (\cite{Reider}) says that $|K_S+3L|$
gives a birational map. We are done.
\end{proof}

\section{\bf Appendix: Anti-plurigenus}

Here we use all those formulae and inequalities and keep the same
notation as in \cite[Section 2]{C-C}. The following result was used
to prove our main results in this paper.

\begin{cor}\label{ap} Let $X$ be a weak $\bQ$-Fano 3-fold. Then
$P_{-m}>0$ for all $m\geq 6$.
\end{cor}
\begin{proof}  Denote by $B:=B(X)$, the virtual basket of singularities of Reid \cite{YPG}.
We shall study the formal basket ${\bf B}:=\{B, P_{-1}\}$ which was
defined in \cite[Section 2]{C-C}. Then we know, from
\cite[2.3]{C-C}, that $-K^3({\bf B})=-K_X^3>0$ and $P_{-m}({\bf
B})=P_{-m}(X)$ for all $m\geq 1$.
\medskip

{\bf Case 1}. Assume $P_{-2}>0$.

If $P_{-5}>0$, then the statement is naturally true. So we only need
to study the situation $P_{-5}=0$. It then follows that
$P_{-1}=P_{-3}=0$.

Now $\epsilon_5\geq 0$ gives $2+P_{-2}\geq 2P_{-4}+\sigma_5$. Since,
by inequality \cite[(2.3)]{C-C}, $P_{-4}\geq 2P_{-2}-1$, we see
$P_{-2}=P_{-4}=1$ and $\sigma_5\leq 1$.

First, consider the case $\sigma_5=0$. Then $\epsilon_6=1\neq 0$, a
contradiction. Thus we must have $\sigma_5=1$. Then we have
$\epsilon_5=0$ and
$$B^{(0)}=B^{(5)}=\{9\times (1,2), (1,3), (1,r)\}, \ r\geq 5.$$
Clearly, since $B^{(5)}$ admits no further prime packings, we have
$B=B^{(5)}=B^{(0)}$. Now inequality \cite[(2.1)]{C-C} implies $r\geq
7$, while inequality \cite[(2.2)]{C-C} gives $r\leq 7$. We see
$r=7$. Now our calculation shows: $-K^3=\frac{1}{42}$,
$P_{-6}=P_{-7}=2$, $P_{-8}=P_{-9}=4$, $P_{-10}=P_{-11}=6$ and
$P_{-12}=9$. Clearly we can see $P_{-m}>0$ for all $m\geq 6$.
\medskip

{\bf Case 2}. Assume $P_{-2}=0$.

According to our complete classification in \cite[Theorem 3.5, Table
A]{C-C}, there are exactly 23 cases as listed there. By the direct
calculation, we know $P_{-m}>0$ for all $6\leq m\leq 12$. Thus it is
easy to deduce $P_{-m}>0$ for all $m\geq 6$. We are done.
\end{proof}


\begin{thebibliography}{99}

\bibitem{Alex94} V. Alexeev, {\em General elephants of $\bQ$-Fano
3-folds}, Compositio Math. {\bf 91} (1994), 91-116.

\bibitem{Ando} T. Ando, {\em Pluricanonical systems of algebraic
varieties of general type of dimension $\leq 5$}. Algebraic
geometry, Sendai, 1985, 1-10, Adv. Stud. Pure Math., 10,
North-Holland, Amsterdam, 1987.


\bibitem{C-C} J. A. Chen, M. Chen, {\em An optimal boundedness on
weak $\bQ$-Fano 3-folds}, Advances in Math. {\bf 219} (2008),
2086-2104.

\bibitem{Camb} M. Chen, {\em Canonical stability in terms of
singularity index for algebraic threefolds}, Math. Proc. Camb. Phil.
Soc. {\bf 13} (2001), 241-264.

\bibitem{Hartshorne} R. Hartshorne, {\em Algebraic Geometry},
Graduate Texts in Mathematics {\bf 52}, Springer-Verlag, New York,
Berlin, Heidelberg, 1977.

\bibitem{I1} V. A. Iskovskih, {\em Fano 3-folds, I}, Izv. Akad. Nauk
SSSR Ser. Mat {\bf 41} (1977). English translation in Math. USSR
izv. {\bf 11} (1977), 485-527.

\bibitem{I2} V. A. Iskovskih, {\em Fano 3-folds, II}, Izv. Akad.
Nauk SSSR Ser. Mat {\bf 42} (1978), English translation in Math.
USSR izv. {\bf 12} (1978), 469-506.

\bibitem{I-P} V. A. Iskovskikh, Yu. G. Prokhorov, {\em Fano varieties}.
Algebraic geometry, V, 1-247, Encyclopaedia Math. Sci., 47,
Springer, Berlin, 1999.



\bibitem{KA} Y. Kawamata, {\em Boundedness of $\bQ$-Fano 3-folds}.
Proceedings of the International Conference on Algebra, Part 3
(Novosibirsk, 1989), 439--445, Contemp. Math., 131, Part 3, Amer.
Math. Soc., Providence, RI, 1992.

\bibitem{KV} Y. Kawamata, {\em A generalization of Kodaira-Ramanujam's
vanishing theorem}, Math. Ann. {\bf 261}(1982), 43-46.



\bibitem{KMM} Y. Kawamata, K. Matsuda, K. Matsuki, {\em Introduction
to the Minimal Model Problem}, Advanced Studies in Pure Mathematics
{\bf 10}, 1987. Algebraic Geometry, Sendai, 1985, pp. 283-360.

\bibitem{KMMT} J. Koll\'ar, Y. Miyaoka, S. Mori, H. Takagi, {\em
Boundedness of canonical $\bQ$-Fano 3-folds}, Proc. Japan Acad. {\bf
76}, Ser. A (2000), 73-77.
\bibitem{K-M} J. Koll\'ar, S. Mori, Birational geometry of algebraic
varieties, 1998, Cambridge Univ. Press.

\bibitem{MM} S. Mori, S. Mukai, {\em Classification of Fano 3-folds
with $b_2\geq 2$}, Algebraic and Topological Theories (Kinosaki,
1984), Kinokuniya, Tokyo, 1985, pp. 496-545.

\bibitem{Mu1} S. Mukai, {\em Curve and symmetric spaces, I}, Amer.
J. Math. {\bf 117} (1995), 1627-1644.

\bibitem{Mu2} S. Mukai, {\em New development of the theory of Fano
threefolds: vector bundle method and moduli problem}, Suguku {\bf
47} (1995), 125-144, English translation in Sugaku Expositions {\bf
15} (2002), 125-150.



\bibitem{Naka} N. Nakayama, {\em Zariski-decomposition and abundance}. MSJ
Memoirs, 14. Mathematical Society of Japan, Tokyo, 2004. xiv+277 pp.


\bibitem{Nami} Y. Namikawa, {\em Smoothing Fano $3$-folds}. J. Algebraic Geom.
{\bf 6} (1997), no. 2, 307-324.



\bibitem{OSS} C. Okonek, M. Schneider, H. Spindler, {\em Vector
bundles on complex projective spaces}, Progress in Math., 3,
Birkhauser, Basel, 1978


\bibitem{Prok} Yu. G. Prokhorov, {\em The degree of $\bQ$-Fano
threefolds}, Sbornik: Mathematics {\bf 198}:11 (2007), 1683-1702.

\bibitem{Reid83} M. Reid, {\em Minimal models of
canonical 3-folds}, Adv. Stud. Pure Math. {\bf 1}(1983), 131-180.

\bibitem{YPG} M. Reid,  {\em Young person's guide to canonical
singularities}, Proc. Symposia in pure Math. {\bf 46} (1987),
345-414.

\bibitem{Reider} I. Reider, {\em Vector bundles of rank 2 and linear systems on
algebraic surfaces}, Ann. Math. {\bf 127}(1988), 309-316.


\bibitem{Shok} V. V. Shokurov, {\em The existence of a straight line
on Fano 3-fold}, Izv. Akad. Nauk SSSR Ser. Mat {\bf 43} (1979),
921-963, English translation in Math. USSR Izv. {\bf 15} (1980),
173-209.

\bibitem{Sho} V. V. Shokurov, {\em 3-fold log flips}, Izv. Russ. A. N.
Ser. Mat. {\bf 56} (1992), 105-203.

\bibitem{Sho2} V. V. Shokurov, {\em Smoothness of a general
anticanonical divisor on a Fano variety}, Izv. Akad. Nauk SSSR, {\bf
14} (1980), 395-405.

\bibitem{Suzuki} K. Suzuki, {\em On Fano indices of $\bQ$-Fano 3-folds}.
Manuscripta Math. {\bf 114} (2004), no. 2, 229--246

\bibitem{TKG} H. Takagi, {\em Classification of primary $\bQ$-Fano
threefolds with anti-canonical Du Val K3 surfaces, I}, J. Alg. Geom.
{\bf 15} (2006), 31-85.

\bibitem{T} S. G. Tankeev, {\em On n-dimensional canonically
polarized varieties and varieties of fundamental type}, Izv. A. N.
SSSR, Ser. Math. {\bf 35}(1971), 31-44.

\bibitem{V} E. Viehweg, {\em Vanishing theorems}, J. reine angew. Math. {\bf
335}(1982), 1-8.

\bibitem{ZQ} Q. Zhang, {\em Rational connectedness of log $Q$-Fano varieties}. J. Reine Angew.
Math. {\bf 590} (2006), 131-142.

\end{thebibliography}
\end{document}